\documentclass[VANCOUVER,STIX1COL]{WileyNJD-v2}
\usepackage{moreverb}
\usepackage{color}
\newcommand{\red}{}

\articletype{Research Article}%

\received{}
\revised{}
\accepted{}

\usepackage{setspace,amsmath,amsthm,amsfonts,mathtools,bm,algorithm,multirow,subcaption,threeparttable}

\newtheorem{Remark}{Remark}
\numberwithin{Theorem}{section}
\numberwithin{Definition}{section}
\numberwithin{Lemma}{section}
\numberwithin{Algorithm}{section}
\numberwithin{equation}{section}
\newcommand{\Tr}{\,\text{Tr}}
\newcommand{\blue}{}
\renewcommand{\thealgorithm}{IP--PMM}

\begin{document}
	\title{A New Preconditioning Approach \\ for an Interior Point--Proximal Method of Multipliers \\ for Linear and Convex Quadratic Programming}
\author[1]{Luca Bergamaschi*}
\author[2]{Jacek Gondzio}
\author[3]{\'Angeles Mart\'{\i}nez}
\author[2]{John W. Pearson}
\author[2]{Spyridon Pougkakiotis}
\authormark{A New Preconditioning Approach for Linear and Convex Quadratic Programming}

\address[1]{\orgdiv{Department of Civil Environmental and Architectural Engineering}, \orgname{University of Padova}, \orgaddress{\country{Italy}}}

\address[2]{\orgdiv{School of Mathematics}, \orgname{University of Edinburgh}, \orgaddress{ \country{UK}}}

\address[3]{\orgdiv{Department of Mathematics and Earth Sciences}, \orgname{University of Trieste}, 
\orgaddress{\country{Italy}}}

\corres{*Luca Bergamaschi, 
Department of Civil Environmental and Architectural Engineering, University of Padova,
Via Marzolo 9, 35100 Padova,  Italy. \email{luca.bergamaschi@unipd.it}}

				\abstract[Summary]{
					In this paper, we address the efficient numerical solution
					of linear and quadratic programming problems, often of large scale.
					With this aim, we devise an infeasible interior point method, blended with 
					the proximal method of multipliers, which in turn
					results in a primal-dual regularized interior point method.
					Application of this method gives rise to a
					sequence of increasingly ill-conditioned 
					linear systems which cannot always be solved
					by factorization methods, due to memory and CPU time restrictions. 
					We propose a novel preconditioning strategy which is based
					on a suitable sparsification of the normal equations matrix in the linear case,
					and also constitutes the foundation of a block-diagonal preconditioner to accelerate
					MINRES for linear systems arising from the solution of general quadratic programming problems.
					Numerical results for a range of test problems 
					demonstrate the robustness of the proposed preconditioning strategy, 
					together with its ability to solve linear systems of very large dimension.}

%

\maketitle

\section{Introduction}
\par In this paper, we consider linear and quadratic programming (LP and QP) problems of the following form:
\begin{equation}
\label{LPQP} \min_{x} \ \big( c^Tx + \frac{1}{2}x^T Q x \big), \ \ \text{s.t.}  \  Ax = b,   \ x \geq 0,
\end{equation}
\noindent where $c,x \in \mathbb{R}^n$, $b \in \mathbb{R}^m$, $A \in \mathbb{R}^{m\times n}$. For quadratic programming problems we have that $Q \succeq 0 \in \mathbb{R}^{n \times n}$, while for linear programming $Q=0$. The problem \eqref{LPQP} is often referred to as the primal form of the quadratic programming problem; the dual form of the problem is given by
\begin{equation}
\label{Dual} \text{max}_{x,y,z}  \ \big(b^Ty - \frac{1}{2}x^T Q x\big), \ \ \text{s.t.}\  -Qx + A^Ty + z = c,\ z\geq 0,
\end{equation}
\noindent where $z \in \mathbb{R}^n$, $y \in \mathbb{R}^m$. Problems of linear or quadratic programming form are fundamental problems in optimization, and arise in a wide range of scientific applications.

\par A variety of optimization methods exist for solving the problem \eqref{LPQP}. Two popular and successful approaches are interior point methods (IPMs) and proximal methods of multipliers (PMMs). Within an IPM, a Lagrangian is constructed involving the objective function and the equality constraints of \eqref{LPQP}, to which a logarithmic barrier function is then added in place of the inequality constraints. Hence, a logarithmic barrier sub-problem is solved at each iteration of the algorithm (see \cite{EJOR-Gon} for a survey on IPMs). The key feature of a PMM is that, at each iteration, one seeks the minimum of the problem \eqref{LPQP} as stated, but one adds to the objective function a penalty term involving the norm of the difference between $x$ and the previously computed estimate. Then, an augmented Lagrangian method is applied to approximately solve each such sub-problem (see \cite{Proximal,SIAMCO-Roc} for a review of proximal point methods, and \cite{AS-Ber,JOTA-Hes,AP-Pow,MOR-Roc} for a review of augmented Lagrangian methods). In this paper we consider a blend of an infeasible IPM and a PMM, which can itself be thought of as a primal-dual regularized IPM. We refer to \cite{arxiv-PouGon} for a derivation of this approach as well as a proof of polynomial complexity. There are substantial advantages of applying regularization within IPMs, and the reliability and fast convergence of the hybrid IP--PMM make it an attractive approach for tackling linear and quadratic programming problems.

\par Upon applying such a technique, the vast majority of the computational effort arises from the solution of the resulting linear systems of equations at each IP--PMM iteration. These linear equations can be tackled in the form of an augmented system, or the reduced normal equations: we focus much of our attention on the augmented system, as unless $Q$ has some convenient structure it is highly undesirable to form the normal equations or apply the resulting matrix within a solver. Within the linear algebra community, direct methods are popular for solving such systems due to their generalizability, however if the matrix system becomes sufficiently large the storage and/or operation costs can rapidly become excessive, depending on the computer architecture used. The application of iterative methods, for instance those based around Krylov subspace methods such as the Conjugate Gradient method (CG) \cite{cg} or MINRES \cite{minres}, is an attractive alternative, but if one cannot construct suitable preconditioners which can be applied within such solvers then convergence can be prohibitively slow, and indeed it is possible that convergence is not achieved at all. The development of powerful preconditioners is therefore crucial.

\par A range of general preconditioners have been proposed for augmented systems arising from optimization problems, see \cite{iter:BGVZ,iter:BGZ,iter:Chai_Toh,iter:DR-pcgQP,iter:SWW-precond,iter:BCO-COAP,iter:OS-pcg1} for instance. However, as is the case within the field of preconditioning in general, these are typically sensitive to changes in structure of the matrices involved, and can have substantial memory requirements. Preconditioners have also been successfully devised for specific classes of programming problems solved using similar optimization methods: applications include those arising from multicommodity network flow problems \cite{iter:Castro}, stochastic programming problems \cite{CLZ16}, formulations within which the constraint matrix has primal block-angular structure \cite{iter:CC}, and PDE-constrained optimization problems \cite{PG17,PPS18}. However, such preconditioners exploit particular structures arising from specific applications; unless there exists such a structure which hints as to the appropriate way to develop a solver, the design of bespoke preconditioners remains a challenge.
%
%

\par It is therefore clear that a completely robust preconditioner for linear and quadratic programming does not currently exist, as available preconditioners are either problem-sensitive (with a possibility of failure when problem parameters or structures are modified), or are tailored towards specific classes of problems. This paper therefore aims to provide a first step towards the construction of generalizable preconditioners for linear and quadratic programming problems. A particular incentive for this work is so that, when new application areas arise that require the solution of large-scale matrix systems, the preconditioning strategy proposed here could form the basis of a fast and feasible solver.

\par This paper is structured as follows. In Section \ref{section Algorithmic Framework} we describe the IP--PMM approach used to tackle linear and quadratic programming problems, and outline our preconditioning approach. In Section \ref{Spectral Analysis} we carry out spectral analysis for the resulting preconditioned matrix systems. In Section \ref{section implementation} we describe the implementation details of the method, and in Section \ref{Section Numerical Results} we present numerical results obtained using the inexact IP--PMM approach. In particular we present the results of our preconditioned iterative methods, and demonstrate that our new solvers lead to rapid and robust convergence for a wide class of problems. Finally, in Section \ref{Conclusion} we give some concluding remarks.

\par \underline{\textbf{Notation:}}
For the rest of this manuscript, superscripts of a vector (or matrix, respectively) will denote the respective components of the vector, i.e. $x^j$ (or $M^{(i,j)}$, respectively). Given a set (or two sets) of indices $\mathcal{I}$ (or $\mathcal{I},\  \mathcal{J}$), the respective sub-vector (sub-matrix), will be denoted as $x^{\mathcal{I}}$ (or $M^{(\mathcal{I},\mathcal{J})}$). Furthermore, the $j$-th row (or column) of a matrix $M$ is denoted as $M^{(j,:)}$ ($M^{(:,j)}$, respectively). Given an arbitrary square (or rectangular) matrix $M$, then $\lambda_{\max}(M)$ and $\lambda_{\min}(M)$ (or $\sigma_{\max}(M)$ and $\sigma_{\min}(M)$) denote the largest and smallest eigenvalues (or singular values) of the matrix $M$, respectively. Given a {\red symmetric}  matrix $M$ we denote as $q(M)$ its {Rayleigh Quotient,
defined as \[ q(M) =  \left \{ z \in \mathbb{R} \ \text{such that} \ z = \frac{x^T M x}{x^T x}, \text{for some} \ x \in \mathbb{R}^n, x \ne 0 \right\}. \]
 Given a square matrix $M \in \mathbb{R}^{n\times n}$, $\textnormal{diag}(M)$ denotes the diagonal matrix satisfying $(\textnormal{diag}(M))^{(i,i)} = M^{(i,i)}$, for all $i \in \{1,\ldots,n\}$. Finally, given a vector $x \in \mathbb{R}^n$, we denote by $X$ the diagonal matrix satisfying $X^{(i,i)} = x^i$, for all $i \in \{1,\ldots,n\}$.}
\section{Algorithmic Framework} \label{section Algorithmic Framework}

\par In this section we derive a counterpart of the Interior Point--Proximal Method of Multipliers (IP--PMM) presented in \cite{arxiv-PouGon} for solving the pair \eqref{LPQP}--\eqref{Dual}, that employs a Krylov subspace method for solving the associated linear systems. For a polynomial convergence result of the method, in the case where the linear systems are solved exactly, the reader is referred to \cite{arxiv-PouGon}. Effectively, we merge the proximal method of multipliers with an infeasible interior point method, and present suitable general purpose preconditioners, using which we can solve the resulting Newton system, at every iteration, by employing an appropriate Krylov subspace method.

\par Assume that, at some iteration $k$ of the method, we have available an estimate $\eta_k$ for the optimal Lagrange multiplier vector $y^*$, corresponding to the equality constraints of \eqref{LPQP}.  Similarly, we denote by $\zeta_k$ the estimate of the primal solution $x^*$. Next, we define the proximal penalty function that has to be minimized at the $k$-th iteration of proximal method of multipliers, for solving (\ref{LPQP}), given the estimates $\eta_k,\ \zeta_k$:
\begin{equation*}
\mathcal{L}^{PMM}_{\delta_k, \rho_k} (x; \eta_k, \zeta_k) = c^Tx + \frac{1}{2}x^T Q x -\eta_k^T (Ax - b) + \frac{1}{2\delta_k}\|Ax-b\|^2 + \frac{\rho_k}{2}\|x-\zeta_k\|^2,
\end{equation*}
\noindent with $\delta_k > 0,\ \rho_k > 0$ some non-increasing penalty parameters. In order to solve the PMM sub-problem, we will apply one (or a few) iterations of an infeasible IPM. To do that, we alter the previous penalty function, by including logarithmic barriers, that is:
\begin{equation} \label{Proximal IPM Penalty}
\mathcal{L}^{IP-PMM}_{\delta_k, \rho_k} (x; \eta_k, \zeta_k) = \mathcal{L}^{PMM}_{\delta_k, \rho_k} (x; \eta_k, \zeta_k) - \mu_k \sum_{j=1}^n \ln x^j,
\end{equation}
\noindent where $\mu_k > 0$ is the barrier parameter. In order to form the optimality conditions of this sub-problem, we equate the gradient of $\mathcal{L}^{IP-PMM}_{\delta_k, \rho_k}(\ \cdot\ ; \eta_k, \zeta_k)$ to the zero vector, i.e.:
\begin{equation*}
c + Qx - A^T \eta_k + \frac{1}{\delta_k}A^T(Ax - b) + \rho_k (x - \zeta_k) - \mu_k X^{-1}\mathbf{1}_n = 0,
\end{equation*}
\noindent where $\mathbf{1}_n$ is a vector of ones of size $n$, and $X$ is a diagonal matrix containing the entries of $x$. We define the variables $y = \eta_k - \frac{1}{\delta_k}(Ax - b)$ and $z = \mu_k X^{-1}\mathbf{1}_n$, to obtain the following (equivalent) system of equations:
\begin{equation}\label{Proximal IPM FOC} \begin{bmatrix}
c + Qx - A^T y - z + \rho_k (x-\zeta_k)\\
Ax + \delta_k (y - \eta_k) - b\\
X z -  \mu_k \mathbf{1}_n
\end{bmatrix} = \begin{bmatrix}
0\\
0\\
0
\end{bmatrix}.
\end{equation}
\par To solve the previous mildly nonlinear system of equations, at every iteration $k$, we employ Newton's method and alter its right-hand side, using a centering parameter $\sigma_k \in  (0,1)$. {The centering parameter determines how fast $\mu_k$ is reduced. For $\sigma_k = 1$, we attempt to find a well-centered solution, while for $\sigma_k = 0$ we attempt to solve directly the original problem \eqref{LPQP}. As this is a crucial parameter in practice, it is often substituted by a predictor--corrector scheme which attempts to accelerate the convergence of the IPM (see \cite{SIAMOpt-Meh,COA-Gon}). For brevity, we present the former approach of heuristically choosing $\sigma_k$, but later on (in Section \ref{section implementation}) present the implemented predictor--corrector scheme.} In other words, at every iteration of IP--PMM we have available an iteration triple $(x_k,y_k,z_k)$ and we wish to solve the following system of equations:
\begin{equation}\label{Newton system}
\begin{bmatrix}
-(Q  + \rho_k I_n)  & A^T & I_n\\
A & \delta_k I_m  & 0\\
Z_k & 0 & X_k
\end{bmatrix} \begin{bmatrix}
\Delta x_k\\
\Delta y_k\\
 \Delta z_k
\end{bmatrix} = \begin{bmatrix}
c + Qx_k - A^T y_k + \sigma_k\rho_k (x_k - \zeta_k)-  z_k \\
b - Ax_k - \sigma_k \delta_k (y_k - \eta_k)\\
 \sigma_k \mu_k \mathbf{1}_n -X_k z_k 
\end{bmatrix} = \begin{bmatrix}
r_{d_k}\\
r_{p_k}\\
r_{\mu_k}
\end{bmatrix}.
\end{equation}
\noindent We proceed by eliminating variables $\Delta z_k$. In particular, we have that:
$$\Delta z_k = X_k^{-1}(r_{\mu_k} - Z_k \Delta x_k),$$
\noindent where $Z_k$ is a diagonal matrix containing the entries of $z_k$. Then, the augmented system that has to be solved at every iteration of IP--PMM reads as follows:
\begin{equation}\label{regularized augmented system}
\begin{bmatrix}
-(Q + \Theta_k^{-1} + \rho_k I_n)  & A^T\\
A & \delta_k I_m 
\end{bmatrix} \begin{bmatrix}
\Delta x_k\\
\Delta y_k
\end{bmatrix} = \begin{bmatrix}
r_{d_k} + z_k - \sigma_k \mu_k X_k^{-1} \mathbf{1}_n \\
r_{p_k}
\end{bmatrix},
\end{equation}
\noindent where $\Theta_k = X_k Z_k^{-1}$. An important feature of the matrix $\Theta_k$ is that, as the method approaches an optimal solution, the positive diagonal matrix has some entries that (numerically) approach infinity, while others approach zero. {By observing the matrix in \eqref{regularized augmented system}, we can immediately see the benefits of using regularization in IPMs. On one hand, the dual regularization parameter $\delta_k$ ensures that the system matrix in \eqref{regularized augmented system} is invertible, even if $A$ is rank-deficient. On the other hand, the primal regularization parameter $\rho_k$ controls the worst-case conditioning of the $(1,1)$ block of \eqref{regularized augmented system}, improving the numerical stability of the method (and hence its robustness). We refer the reader to \cite{arxiv-PouGon,JOTA-PouGon,OMS-AltGon} for a review of the benefits of regularization in the context of IPMs.}

\par As we argue in the spectral analysis, in the case where $Q = 0$, or $Q$ is diagonal, it is often beneficial to form the normal equations and approximately solve them using preconditioned CG. Otherwise, we solve system \eqref{regularized augmented system} using preconditioned MINRES. The normal equations read as follows:
\begin{equation} \label{regularized normal equations}
	M_{NE,k}  \Delta y_k = \xi_k, \qquad  M_{NE,k} = A (\Theta_k^{-1} + Q + \rho_k I_n)^{-1} A^T + \delta_k I_m ,
\end{equation}
\noindent where 
$$\xi_k = r_{p_k} + A(Q + \Theta_k^{-1} + \rho_k I_n)^{-1}(r_{d_k} + z_k - \sigma_k \mu_k X_k^{-1} \mathbf{1}_n). $$

\par In order to employ preconditioned MINRES or CG to solve \eqref{regularized augmented system} or \eqref{regularized normal equations} respectively, we must find an approximation for the coefficient matrix in \eqref{regularized normal equations}. To do so, we employ a symmetric and positive definite block-diagonal preconditioner for the saddle-point system \eqref{regularized augmented system}, involving approximations for the negative of the (1,1) block, as well as the Schur complement $M_{NE}$. See \cite{Kuzn95,MGW00,SIAMNA-SilWat} for motivation of such saddle-point preconditioners. 
In light of this, we approximate $Q$ in the (1,1) block by its diagonal, i.e. $\tilde{Q}  = \text{diag}(Q)$. 

	Then, we define the diagonal
matrix $E_k$ with entries
\begin{equation}
	\label{defE}
	E_k^{(i,i)} = \begin{cases} 0 & \ \text{if}  \ 
		\big((\Theta_k^{(i,i)})^{-1} + \tilde{Q}^{(i,i)} + \rho_k\big)^{-1} < C_{E,k} \min\{\mu_k, 1\}, \\
	\big((\Theta_k^{(i,i)})^{-1} + \tilde{Q}^{(i,i)} + \rho_k\big)^{-1} & \ \text{otherwise,} \end{cases} \end{equation}
where $i \in \{1,\ldots,n\}$, $C_{E,k}$ is a constant, and we construct the normal equations approximation $P_{NE,k} = L_M L_M^T$, by computing the (exact) Cholesky factorization of 
\begin{equation} \label{normal equations preconditioner}
P_{NE,k} = A E_k A^T + \delta_k I_m.
\end{equation}
	{\red 
	The dropping threshold in (\ref{defE}) guarantees that a coefficient in 
	the diagonal matrix 
		$\left(\Theta_k^{-1} + \tilde{Q}+ \rho_kI^{-1}\right)^{-1}$ is set to zero only if it is
		below a  constant times the barrier parameter $\mu_k$. }\textcolor{black}{As a consequence fewer outer products of columns of $A$ contribute to the normal equations, and the resulting preconditioner $P_{NE,k}$ is expected to be more sparse than $M_{NE,k}$.} {This choice is also crucial to guarantee
		that the eigenvalues of the preconditioned normal equations matrix  are independent
		of $\mu$. Before discussing the role of the \textit{constant} $C_{E,k}$, let us first address the preconditioning of the augmented system matrix in \eqref{regularized augmented system}. } The matrix $P_{NE,k}$ acts as a preconditioner for CG applied to the normal equations. In order to construct a preconditioner for the augmented system matrix in \eqref{regularized augmented system}, we employ a block-diagonal preconditioner of the form:

\begin{equation} \label{augmented system preconditioner}
P_{AS,k} = 
	\begin{bmatrix} 
\tilde{Q} + \Theta_k^{-1} + \rho_k I_n & 0 \\
0 & P_{NE,k}
\end{bmatrix},
\end{equation}
	\noindent with $P_{NE,k}$ defined in \eqref{normal equations preconditioner}. Note that MINRES requires a symmetric positive definite preconditioner and hence many other block preconditioners for \eqref{regularized augmented system} are not applicable. For example, block-triangular preconditioners, motivated by the work in \cite{Ipsen01,MGW00}, would generally require a non-symmetric solver such as GMRES \cite{SaSc86}. Nevertheless, block-diagonal preconditioners have been shown to be very effective in practice for problems with the block structure of \eqref{regularized augmented system} (see for example \cite{AN-BenGolLie,SIMAX-Not,SIAMNA-SilWat}). Furthermore, it can often be beneficial to employ CG with the preconditioner \eqref{normal equations preconditioner}, in the case where $Q=0$ or $Q$ is diagonal, since the former is expected to converge faster than MINRES with \eqref{augmented system preconditioner}. This will become clearer in the next section, where eigenvalue bounds for each of the preconditioned matrices are provided. 
	\par In view of the previous discussion, we observe that the quality of both preconditioners heavily depends on the choice of constant $C_{E,k}$, since this constant determines the quality of the approximation of the normal equations using \eqref{normal equations preconditioner}. In our implementation this constant is tuned dynamically, based on the quality of the preconditioner and its required memory (see Section \ref{section implementation}). {\red Moreover}, following the developments in \cite{arxiv-PouGon}, we tune the regularization variables $\delta_k,\ \rho_k$ based on the barrier parameter $\mu_k$. In particular, $\delta_k,\ \rho_k$ are forced to decrease at the same rate as $\mu_k$. The exact updates of these parameters are presented in Section \ref{section implementation}. As we will show in the next section, this tuning choice is numerically beneficial, since if $\delta_k,\ \rho_k$ are of the same order as $\mu_k$, then the spectrum of the preconditioned normal equations is independent of $\mu_k$; a very desirable property for preconditioned systems arising from IPMs.

\section{Spectral Analysis} \label{Spectral Analysis}
\subsection{Preconditioned normal equations}
	
In this section we provide a spectral analysis of the preconditioned normal equations in the LP or separable QP
	case, assuming that \eqref{normal equations preconditioner} is used as the preconditioner. Although this is a specialized setting, we may make use of the following result in our analysis of the augmented system arising from the general QP case. 
	\par Let us define this normal equations matrix $\tilde{M}_{NE,k}$, as 
\begin{equation} \label{regularized NE, Compact form}
	\tilde{M}_{NE,k} = A \tilde{G}_k  A^T + \delta_k I_m, \quad \text{with} \ 
	\tilde{G}_k =  \left(\tilde Q + \Theta^{-1} + \rho_k I_n\right)^{-1}. 
\end{equation}


	\par The following Theorem provides lower and upper bounds on the eigenvalues of $P_{NE,k}^{-1} \tilde{M}_{NE,k}$, at an arbitrary iteration $k$ of Algorithm IP--PMM.
 \begin{theorem}
	 \label{theorem_LP}
	 There are $m-r$ eigenvalues of $P_{NE,k}^{-1} \tilde{M}_{NE,k}$ at one, where $r$ is the column rank of $A^T$,
	 corresponding to  linearly
	 independent vectors belonging to the nullspace of $A^T$.
	 The remaining eigenvalues are bounded as 
	 \[1 \le \lambda \le 1 + \frac{C_{E,k} \mu_k}{\delta_k} \sigma^2_{\max} (A). \]
 \end{theorem}
 \begin{proof}
	 The eigenvalues of $P_{NE,k}^{-1} \tilde{M}_{NE,k}$ must satisfy
	 \begin{equation}
		 A \tilde{G}_k A^T u + \delta_k u =  
	  \lambda A E_k A^T u + \lambda \delta_k u . 
		 \label{genprob}
	 \end{equation}
	 Multiplying (\ref{genprob}) on the left by $u^T$ and setting $z = A^T u$ yields
	 \[ \lambda = \frac{z^T \tilde{G}_k z + \delta_k \|u\|^2}
{z^T E_k z + \delta_k \|u\|^2} = 1 + 
	 \frac{z^T \left(\tilde{G}_k-E_k\right) z}
	 {z^T E_k z + \delta_k \|u\|^2} =  1 + \alpha. \]
	 For every vector $u$ in the nullspace of $A^T$ we have $z = 0$ and $\lambda = 1$.
	 The fact that both $E_k$ and $\tilde{G}_k - E_k \succeq 0$ (from the definition of $E_k$) implies the lower bound.
	 To prove the upper bound we first observe that  $\lambda_{\max} (\tilde{G}_k-E_k) \le C_{E,k} \mu_k$;
	 then
	 \[ \alpha = 
	 \frac{z^T \left(\tilde{G}_k-E_k\right) z}
	 {z^T E_k z + \delta_k \|u\|^2} \le
	 \frac{z^T \left(\tilde{G}_k-E_k\right) z}
	 {\delta_k \|u\|^2}  =
	 \frac{z^T \left(\tilde{G}_k-E_k\right) z}{\|z\|^2} \frac{1}{\delta_k} \frac{\|z\|^2}{\|u\|^2}
	 = \frac{z^T \left(\tilde{G}_k-E_k\right) z}{\|z\|^2} \frac{1}{\delta_k} \frac{u^T A A^T u}{\|u\|^2},
	 \]
	 and the thesis follows by inspecting the Rayleigh Quotients of $\tilde{G}_k - E_k$ and
	 $A A^T$.
 \end{proof}

\begin{Remark}
Following the discussion in the end of the previous section, we know that $\dfrac{\mu_k}{\delta_k} = O(1)$, since IP--PMM forces $\delta_k$ to decrease at the same rate as $\mu_k$. Combining this with the result of Theorem \ref{theorem_LP} implies that the condition number of the preconditioned normal equations is asymptotically independent of $\mu_k$.
\end{Remark}
\begin{Remark}
	 In the LP case ($ Q = 0$), or the separable QP case ($Q$ diagonal),
	 Theorem \ref{theorem_LP} characterizes the eigenvalues of the preconditioned matrix within the
	 CG method.

\end{Remark}

\subsection{BFGS-like low-rank update of the \texorpdfstring{$\bm{P_{NE,k}}$}{p} preconditioner}
Given a rectangular (tall) matrix $V \in \mathbb{R}^{m \times p} $  with maximum column rank,
it is possible to define a generalized block-tuned preconditioner $P$ satisfying the property
\[ P^{-1} \tilde{M}_{NE,k} V = \nu V, \]
so that the columns of $V$ become eigenvectors of the preconditioned matrix corresponding to the 
eigenvalue $\nu$.
A way to construct $P$ (or its explicit inverse)
is suggested by the BFGS-based preconditioners used e.g. in \cite{bbm08sisc} for accelerating
Newton linear systems or analyzed in \cite{BMnlaa14} for general sequences of linear systems, that is
\[ P^{-1} = \nu V \Pi V^T + (I_m  - V \Pi V^T \tilde{M}_{NE,k}) P_{NE,k}^{-1} (I_m  - \tilde{M}_{NE,k} V \Pi V^T), \quad \text{with} \quad \Pi = (V^T \tilde{M}_{NE,k} V)^{-1}.\]

Note also that if the columns of $V$ would be chosen as e.g. the $p$ 
exact rightmost eigenvectors of $P_{NE,k}^{-1} \tilde{M}_{NE,k}$ (corresponding to the $p$ largest eigenvalues) then all the other eigenpairs, 
\[(\lambda_1, z_1), \ldots, (\lambda_{m-p}, z_{m-p}),\]
of the new preconditioned matrix $P^{-1} \tilde{M}_{NE,k}$ would remain unchanged {(nonexpansion
of the spectrum of $P^{-1} \tilde{M}_{NE,k}$, see \cite{GraSarTshi})}, 
as stated in the following:
\begin{theorem} \label{thm: low-rank updates}
	If the columns of $V$ are the exact rightmost eigenvectors of $P_{NE,k}^{-1} \tilde{M}_{NE,k}$ then,
for $j = 1, \ldots, m-p$, it holds that
	\[ P^{-1} \tilde{M}_{NE,k} z_j = P_{NE,k}^{-1} \tilde{M}_{NE,k} z_j = \lambda_j z_j.\] 
\end{theorem}
\begin{proof}
	The eigenvectors of the symmetric generalized eigenproblem  $\tilde{M}_{NE,k} x = \lambda P_{NE,k} x$ form
	a $P_{NE,k}$-orthonormal basis, and therefore
	$V^T P_{NE,k} z_j = V^T \tilde{M}_{NE,k} z_j = 0, \ j = 1, \ldots, m-p$.  
	Then
	\begin{eqnarray*} P^{-1} \tilde{M}_{NE,k} z_j &= &
		\nu V {\blue \Pi} V^T \tilde{M}_{NE,k} z_j \\ && +\ (I_m - V{\blue \Pi} V^T \tilde{M}_{NE,k}) P_{NE,k}^{-1} (\tilde{M}_{NE,k} z_j  - \tilde{M}_{NE,k} V{\blue \Pi}  V^T \tilde{M}_{NE,k} z_j)   \\
		& = & (I_m - V{\blue \Pi}  V^T \tilde{M}_{NE,k}) P_{NE,k}^{-1} \tilde{M}_{NE,k} z_j=  (I_m - V {\blue \Pi} V^T \tilde{M}_{NE,k}) \lambda_j z_j = \lambda_j z_j. 
\end{eqnarray*}
\end{proof}
\noindent
{\red 
A similar result to Theorem \ref{thm: low-rank updates}  can be found in~\cite{GraSarTshi}, where the low-rank
correction produces what the authors call a \textit{second-level preconditioner}.
}

Usually columns of $V$ are chosen as the (approximate) eigenvectors of $P_{NE,k}^{-1} \tilde{M}_{NE,k}$ corresponding to the smallest
eigenvalues of this matrix~\cite{Saad-et-al2000,LB_Algorithms_2020}. 
However, this choice would not produce a significant reduction 
in the condition number of the preconditioned matrix as
the spectral analysis of Theorem \ref{theorem_LP} suggests a possible clustering of smallest eigenvalues
around 1. We choose instead, as the columns of $V$, the rightmost eigenvectors of $P_{NE,k}^{-1}  \tilde{M}_{NE,k}$, approximated
with low accuracy by the function {\texttt eigs} of MATLAB.
The $\nu$ value  must be selected to satisfy $\lambda_{\min}(P_{NE,k}^{-1}\tilde{M}_{NE,k}) < \nu \ll 
\lambda_{\max}(P_{NE,k}^{-1}\tilde{M}_{NE,k})$. We choose  $\nu = 10$, {\red to ensure that
this new eigenvalue lies in the interior of the spectral interval},
and the column size of $V$ as $p = 10$. {\red This last choice is driven by experimental evidence
that in most cases there are a small number of large outliers in $P_{NE,k}^{-1}\tilde{M}_{NE,k}$. 
A larger value of $p$ would (unnecessarily) increase the cost of applying the preconditioner.} 

Finally, by computing approximately the rightmost eigenvectors,
we would expect a slight perturbation of $\lambda_1, \ldots, \lambda_{m-p}$, depending on the accuracy of this approximation. For a detailed perturbation analysis see e.g. \cite{Tshi}.
\subsection{Preconditioned augmented system}

In the MINRES solution of QP instances the system matrix is
\[ M_{AS,k} = \begin{bmatrix} -F_k & A^T \\ A & \delta_k I_m \end{bmatrix}, \qquad F_k = Q + \Theta_k^{-1} + \rho_k I_n, \]
	while the preconditioner is
\[ P_{AS,k} = \begin{bmatrix} \tilde F_k & 0 \\ 0 & P_{NE,k}  \end{bmatrix}, \qquad \tilde F_k = 
	\tilde Q + \Theta_k^{-1} + \rho_k I_n \equiv \tilde{G}_k^{-1}. \]
	The following Theorem will characterize the eigenvalues of $P_{AS,k}^{-1} M_{AS,k}$
	in terms of the extremal eigenvalues of the preconditioned
	(1,1) block of \eqref{regularized augmented system}, $\tilde F_k^{-1} F_k$, and of
	$P_{NE,k}^{-1} \tilde{M}_{NE,k}$ as described by Theorem \ref{theorem_LP}. We will 
	work with (symmetric positive definite) similarity transformations of these matrices defined as
			\begin{equation}
				\label{simsym}
				\hat F_k =  \tilde F_k^{-1/2} F_k  \tilde F_k^{-1/2}, \quad 
				\hat{M}_{NE,k} = P_{NE,k}^{-1/2} \tilde{M}_{NE,k} P_{NE,k}^{-1/2},
\end{equation}
and set

		\[
		\begin{array}{lcllcllcl}
			\alpha_{NE} &=& \lambda_{\min} ( \hat{M}_{NE,k}),   & \beta_{NE} &=& \lambda_{\max} \left( \hat{M}_{NE,k}\right),
& \kappa_{NE} &=& \dfrac{\beta_{NE}}{\alpha_{NE}}, 
			\\[.6em]
\alpha_{F} &=& \lambda_{\min} \left(\hat F_k\right), &   \beta_{F} &=& \lambda_{\max} \left(\hat F_k\right),
& \kappa_{F} &=& \dfrac{\beta_{F}}{\alpha_{F}}. 
\end{array} \]
\noindent Hence, an arbitrary  element of the numerical range of these matrices is represented as:
\[\begin{array}{lcllcl}
	\gamma_{NE} & \in & q(\hat {M}_{NE,k}) = [\alpha_{NE}, \beta_{NE}],  ~~
	& ~~\gamma_{F}  &\in &q(\hat F_k) =  [\alpha_{F}, \beta_{F}]. 
\end{array}\] 
\noindent Similarly, an arbitrary element of $q(P_{NE,k})$ is denoted by
\[\begin{array}{lcllclcl}
\gamma_{p} & \in & [\lambda_{\min}(P_{NE,k}),\lambda_{\max} (P_{NE,k})] &\subseteq & \left[\delta_k,\dfrac{\sigma_{\max}^2(A)}{\rho_k} + \delta_k\right).
	\end{array}\]
\noindent  Observe that $\alpha_{F} \le 1 \le \beta_{F}$ as 
		\[ \frac 1n \sum_{i=1}^n \lambda_i \left(\tilde F_k^{-1} F_k\right) = \frac 1n \Tr\left(\tilde F_k^{-1} F_k\right)
 = 1.\]
	
		\begin{theorem}
		Let $k$ be an arbitrary iteration of IP--PMM. Then, the eigenvalues of $P_{AS,k}^{-1} M_{AS,k}$ lie in the union of the following intervals:
			\[ I_- = \left[- \beta_{F} -\sqrt{\beta_{NE}}, -\alpha_{F}\right]; \quad I_+ = \left[\frac{1}{1+\beta_{F}} , 1 + \sqrt{\beta_{NE}-1}\right].
			\]
	\end{theorem}
	\begin{proof}
		  The eigenvalues of $P_{AS,k}^{-1} M_{AS,k}$ are the same as those of  
		\[P_{AS,k}^{-1/2} M_{AS,k}P_{AS,k}^{-1/2}  = 
		\begin{bmatrix} \tilde F_k^{-1/2} & 0 \\ 0 & P_{NE,k}^{-1/2}  \end{bmatrix}
\begin{bmatrix} -F_k & A^T \\ A & \delta_k I_m \end{bmatrix}
			\begin{bmatrix} \tilde F_k^{-1/2} & 0 \\ 0 & P_{NE,k}^{-1/2}  \end{bmatrix}
				=
	\begin{bmatrix} -\hat F_k & R_k^T \\ R_k & 
	\delta_k P_{NE,k}^{-1} 	\end{bmatrix},  \]
			where  $\hat F_k$ is defined in (\ref{simsym}) and
			$R_k = P_{NE,k}^{-1/2} A \tilde F_k^{-1/2} .$

Any eigenvalue $\lambda$ of $P_{AS,k}^{-1/2} M_{AS,k}P_{AS,k}^{-1/2} $
		must therefore satisfy
		\begin{eqnarray}
			\label{first}
		-\hat F_k w_1\ +&  R_k^T w_2 &= \lambda w_1, \\
			\label{second}
		R_k w_1\ +&   \delta_k P_{NE,k}^{-1} w_2  &= \lambda w_2. \end{eqnarray}
		First note that
		\begin{equation}
			\label{Definition_of_R_k R_k^T}
			  R_k R_k^T = P_{NE,k}^{-1/2} A \tilde F_k^{-1} A^T P_{NE,k}^{-1/2} = P_{NE,k}^{-1/2} \left(\tilde M_{NE,k} - \delta_k I_m\right) P_{NE,k}^{-1/2} = \hat{M}_{NE,k} -
		\delta_k P_{NE,k}^{-1}. 
		\end{equation}
		The eigenvalues of $R_k R_k^T$ are therefore characterized by Theorem \ref{theorem_LP}.
		If $\lambda  \not \in [-\beta_{F}, -\alpha_{F}]$ then $\hat F_k + \lambda I_n$ is symmetric positive (or negative) definite;
		moreover $R_k^T w_2 \ne 0$.
		Then from (\ref{first}) we obtain an expression for $w_1$: 
		\[ w_1 = (\hat F_k + \lambda I_n)^{-1} R_k^T w_2, \]
		which, after substituting in (\ref{second}) yields
		\[ R_k  (\hat F_k + \lambda I_n)^{-1} R_k^T  w_2  + \delta_k P_{NE,k}^{-1} w_2= \lambda w_2. \]
		Premultiplying by $w_2^T$ and dividing by $\|w_2\|^2$, we obtain the following equation where we set $z = R_k^T w_2$:
		\[ \lambda  = \frac{z^T (\hat F_k + \lambda I_n)^{-1} z}{z^T z}
		\frac {w_2^T R_k R_k^T w_2}{w_2^T w_2}  +\delta_k \frac{w_2^T P_{NE,k}^{-1} w_2}{w_2^Tw_2} = \frac{1}{\gamma_{F} + \lambda} \left(\gamma_{NE} - \frac{\delta_k}{\gamma_p}\right)
		+  \frac{\delta_k}{\gamma_{p}}. \]
		So $\lambda$ must satisfy the following second-order algebraic equation
		\[ \lambda^2 +  \left(\gamma_{F} -\omega \right)\lambda-\left(\omega (\gamma_F - 1)+\gamma_{NE}\right) = 0.\]
		where we have set $\omega = \dfrac{\delta_k}{\gamma_{p}}$ satisfying $\omega \leq 1$ for all $k \geq 0$. 

		We first consider the negative eigenvalue solution of the previous algebraic equation, that is:
		\begin{equation*}
		\begin{split}
		 \lambda_{-} =\ & \frac{1}{2} \bigg[\omega - \gamma_{F} -\sqrt{(\gamma_{F} - \omega)^2 + 4(\omega \gamma_{F} -\omega + \gamma_{NE})} \bigg]\\
			=\ & \frac{1}{2} \bigg[ \omega - \gamma_{F} -\sqrt{(\gamma_{F} + \omega)^2 + 4(\gamma_{NE}-\omega)}\bigg]\\
		\leq\ & \frac{1}{2} \bigg[ \omega - \gamma_{F} - \sqrt{(\gamma_{F} + \omega)^2} \bigg] = -\gamma_{F} \leq - \alpha_{F},
		\end{split} 
		 \end{equation*}
		 {\red where the last line is obtained by noting that $\gamma_{NE}\geq1$ from Theorem \ref{theorem_LP}, and $\omega\leq1$. In order to derive a lower bound on $\lambda_{-}$ we 
		use the fact that $\lambda_{-}$ is an increasing function with respect to $\omega$, and decreasing with respect to $\gamma_{NE}$ and $\gamma_F$.  Hence,}
		 \begin{equation*}
		\begin{split}
			\lambda_{-} =\ & \frac{1}{2} \bigg[ \omega - \gamma_{F} -\sqrt{(\gamma_{F} + \omega)^2 + 4(\gamma_{NE}-\omega)}\bigg]\\
		\geq\ & \frac{1}{2} \bigg[ - \gamma_{F} - \sqrt{\gamma_{F} ^2 + 4\gamma_{NE}}\bigg] \\
		\geq\ & \frac{1}{2}\bigg[-\beta_{F} -  \sqrt{\beta_{F}^2 + 4\beta_{NE}} \bigg] \geq -\beta_{F} - \sqrt{\beta_{NE}}.
		 	\end{split}
		 \end{equation*}
		Combining all the previous yields:
		\[\lambda_- \begin{cases} 
			& \geq -\beta_{F} - \sqrt{\beta_{NE}}, \\[.6em]
			&  \leq -\alpha_{F}.
		 \end{cases}
			\]
			Note that this interval for $\lambda_-$ contains the interval $[-\beta_{F},-\alpha_{F}]$, which we have excluded in order to carry out the analysis.

\par Regarding the positive eigenvalues we have that:
\[\lambda_+ = \frac{1}{2}\bigg[\omega-\gamma_{F} + \sqrt{(\gamma_{F}-\omega)^2+ 4(\omega \gamma_{F} -\omega + \gamma_{NE})}\bigg] =\frac{1}{2}\bigg[\omega-\gamma_{F} + \sqrt{(\gamma_{F}+\omega)^2+ 4(\gamma_{NE}
-\omega) }\bigg].\]
\noindent We proceed by finding a lower bound for $\lambda_{+}$. To that end, we notice that $\lambda_+$ is a decreasing function with respect to the variable $\gamma_{F}$ and increasing with respect to $\gamma_{NE}$. Hence, we have that:
\begin{equation*}
\begin{split}
	\lambda_+ \geq \ & \frac{1}{2}\bigg[\omega - \beta_{F} + \sqrt{(\beta_{F} + \omega)^2 + 4(\alpha_{NE}-\omega)} \bigg]	\\
	\geq \ & \frac{1}{2}\bigg[\omega - \beta_{F} + \sqrt{(\beta_{F} + \omega)^2 + 4(1-\omega)} \bigg],\ ~~\textnormal{since } \alpha_{NE} \geq 1,\textnormal{ from Theorem \ref{theorem_LP}},\\
 \geq \ & \frac{1}{2}\bigg[ -\beta_{F} + \sqrt{\beta_{F}^2 + 4}\bigg],\ ~~\textnormal{since the previous is increasing with respect to } \omega, \\
 \geq \ & \frac{1}{1 + \beta_{F}}. \\
\end{split}
\end{equation*}
\noindent Similarly, in order to derive an upper bound for $\lambda_+$, we observe that $\lambda_+$ is an increasing function with respect to $\omega$, decreasing with respect to $\gamma_{F}$, and increasing with respect to $\gamma_{NE}$. Combining all the previous yields:
\begin{equation*}
\begin{split}
	\lambda_+ \leq \ & \frac{1}{2} \bigg[1 - \alpha_{F} + \sqrt{(\alpha_{F} + 1)^2 + 4(\beta_{NE}-1)} \bigg]
	 \le 1 + \sqrt {\beta_{NE}-1},
\end{split}
\end{equation*}
\noindent where we used the fact that $\omega \leq 1$. Then, combining all the previous gives the desired bounds, that is:		
		\[
		\lambda_+ 
		\begin{cases}  
		&	\geq \dfrac{1}{1+\beta_F} \\[.6em]
		&
	 \le 1 + \sqrt {\beta_{NE}-1},
		\end{cases} \]
\noindent and completes the proof.
	\end{proof}
	\begin{Remark}
		It is well known that a pessimistic bound on the convergence rate of MINRES can be obtained
		if the size of $I_-$ and $I_+$ are roughly the same~{\red \cite{agSIAM}}. In our case, as usually $\beta_F \ll \beta_{NE}$,
		we can assume that the length of both intervals is roughly $\sqrt{\beta_{NE}}$. 
		As a heuristic we may therefore use \textnormal{\cite[Theorem 4.14]{ESW14}}, which predicts the reduction of the residual in the $P_{AS}^{-1}$-norm in the case where both intervals have exactly equal length. This then implies that
		\[ \frac{\|r_k\|}{\|r_0\|} \le 2 \left(\frac{\kappa - 1}{\kappa+1} \right)^{\lfloor k/2 \rfloor},\]
		where 
		\begin{eqnarray*} \kappa &\approx&\frac{1+\beta_F}{\alpha_F} \left(1+\sqrt{\beta_{NE}-1}\right)
		(\beta_F + \sqrt{\beta_{NE}})
			\le  2 \kappa_F \left(\sqrt{1+\beta_{NE}}\right) (\beta_F + \sqrt{\beta_{NE}}) \\
		&\approx& 2 \beta_{NE} \cdot \kappa_F \le 2 \kappa_{NE} \cdot \kappa_F. \end{eqnarray*}
	\end{Remark}
	\begin{Remark} \label{remark: cg vs minres iterations}
		In the LP case $\tilde F_k = F_k$ and therefore $\kappa_F = 1$. It then turns out that
		$\kappa \approx 2\kappa_{NE}$.  The number of MINRES iterations is then
		driven by $ 2 \kappa_{NE}$ while the CG iterations depend on $\sqrt {\kappa_{NE}}$ \textnormal{\cite{LieTic2004}}. We highlight that different norms are used to describe the reduction in the relative residual norm for MINRES and CG.
	\end{Remark}



\section{Algorithms and Implementation Details} \label{section implementation}
\par In this section, we provide some implementation details of the method. The code was written in MATLAB and can be found here: \textbf{https://github.com/spougkakiotis/Inexact\_IP--PMM} (\href{https://github.com/spougkakiotis/Inexact_IP-PMM}{source link}).  In the rest of this manuscript, when referring to CG or MINRES, we implicitly assume that the methods are preconditioned. In particular, the preconditioner given in \eqref{normal equations preconditioner} is employed when using CG, while the preconditioner in \eqref{augmented system preconditioner} is employed when using MINRES.

\subsection{Input problem}
\noindent The method takes input problems of the following form:
\begin{equation*} 
\min_{x} \ \big( c^Tx + \frac{1}{2}x^T Q x \big), \ \ \text{s.t.}  \  Ax = b,   \ x^{I} \geq 0,\ x^{F}\ \text{free},  
\end{equation*}
\noindent where $I = \{1,...,n\} \setminus F$ is the set of indices indicating the non-negative variables. In particular, if a problem instance has only free variables, no logarithmic barrier is employed and the method reduces to a standard proximal method of multipliers.
\par In the pre-processing stage, we check if the constraint matrix is well scaled, i.e. if:
$$\Big(\max_{i \in \{1,...,m\},j \in \{1,...,n\}}(|A^{(i,j)}|) < 10\Big) \wedge \Big(\min_{i \in \{1,...,m\},j \in \{1,...,n\}:\ |A^{(i,j)}| > 0}(|A^{(i,j)}|) > 0.1\Big).$$ If the previous is not satisfied, we apply geometric scaling to the rows of $A$, that is, we multiply each row of $A$ by a scalar of the form:
$$d_i = \frac{1}{\sqrt{\max_{j \in \{1,...,n\}}(|A^{(i,:)}|) \cdot \min_{j \in \{1,...,n\}:\ |A^{(i,j)}| > 0}(|A^{(i,:)}|)}},\ \forall\ i \in \{1,...,m\}.$$

\subsection{Interior point-proximal method of multipliers} \label{section starting point}
\subsubsection{{Parametrization and the Newton system}} 
{Firstly,} in order to construct a reliable starting point for the method, we follow the developments in \cite{SIAMOpt-Meh}. To this end, we try to solve the pair of problems \eqref{LPQP}--\eqref{Dual}, ignoring the non-negativity constraints, which yields 
\begin{equation*} 
\tilde{x} = A^T(AA^T)^{-1}b,\qquad \tilde{y} = (AA^T)^{-1}A(c+Q\tilde{x}), \qquad  \tilde{z} = c - A^T \tilde{y} + Q\tilde{x}.
\end{equation*}
\noindent However, we regularize the matrix $AA^T$ and employ the preconditioned CG method to solve these systems without forming the normal equations. We use the Jacobi preconditioner to accelerate CG, i.e. $P = \text{diag}(AA^T) + \delta I_m$, where $\delta = 8$ is set as the regularization parameter. Then, in order to guarantee positivity and sufficient magnitude of $x_I,z_I$, we shift these components by some appropriate constants. These shift constants are the same as the ones used in the starting point developed in \cite{SIAMOpt-Meh}, and hence are omitted for brevity of presentation. 
\par The Newton step is computed using a predictor--corrector method. We provide the algorithmic scheme in Algorithm \ref{Algorithm Predictor-Corrector}, and the reader is referred to \cite{SIAMOpt-Meh} for a complete presentation of the method. We solve the systems \eqref{Augmented System predictor} and \eqref{Augmented System corrector}, using the proposed preconditioned iterative methods (i.e. CG or MINRES). Note that in case CG is employed, we apply it to the normal equations of each respective system. Since we restrict the maximum number of Krylov iterations, we must also check whether the solution is accurate enough. If it is not, we drop the computed directions and improve our preconditioner. If this happens for 10 consecutive iterations, the algorithm is terminated.
\renewcommand{\thealgorithm}{PC}
\begin{algorithm}[!h]
\caption{Predictor--Corrector Method}
    \label{Algorithm Predictor-Corrector}
\begin{algorithmic}
	\State Compute the predictor:
	\begin{equation} \label{Augmented System predictor}
\begin{bmatrix} 
-(Q+\Theta^{-1}+\rho_k I_n) &   A^T \\
A & \delta_k I_m
\end{bmatrix}
\begin{bmatrix}
\Delta_p x\\ 
\Delta_p y
\end{bmatrix}
= 
\begin{bmatrix}
c + Qx_k - A^T y_k -  \rho_k (x_k - \zeta_k) - d_1\\
b-Ax_k - \delta_k (y_k - \eta_k)
\end{bmatrix},
\end{equation}
\noindent where $d_1^I = -\mu_k (X^I)^{-1}e_{|I|}$ and $d_1^F = 0$ (components of $d_1$ corresponding to inequality constraints and free variables)\color{black}.
\State Retrieve $\Delta_p z$:
$$\Delta_p z^I = d_1^I - (X^I)^{-1}(Z^{I} \Delta_p x^{I}),\ \ \Delta_p z^F = 0.$$
\State Compute the step in the non-negativity orthant:
	\begin{align*} 
\alpha_x^{\max} = \min_{(\Delta_p x^{I(i)} < 0)} \bigg \{1,-\frac{x^{I(i)}}{\Delta_p x^{I(i)}}\bigg\},\ \ \alpha_z^{\max} = \min_{(\Delta_p z^{I(i)} < 0)} \bigg \{1,-\frac{z_k^{I(i)}}{\Delta_p z^{I(i)}}\bigg\},
	\end{align*} 
\noindent for $i = 1,...,|I|$, and set:
\begin{equation*}
\alpha_x = \tau \alpha_x^{\max},\ \ \alpha_z= \tau \alpha_z^{\max},
\end{equation*}
\noindent with $\tau = 0.995$ (avoid going too close to the boundary).
\State Compute a centrality measure: 
$$g_{\alpha} = (x^I+\alpha_x \Delta_p x^I)^T(z^I + \alpha_z \Delta_p z^I).$$
\State Set: $\mu = \big(\frac{g_{\alpha}}{(x^I_k)^Tz^I_k} \big)^2 \frac{g_{\alpha}}{|I|}$
\State Compute the corrector:
\begin{equation} \label{Augmented System corrector}
\begin{bmatrix} 
-(Q+\Theta^{-1}+\rho_k I_n) &   A^T \\
A & \delta_k I_m
\end{bmatrix}
\begin{bmatrix}
\Delta_c x\\ 
\Delta_c y
\end{bmatrix}
= 
\begin{bmatrix}
 d_2\\
0
\end{bmatrix},
\end{equation}
\noindent with $d^I_2 = \mu (X^I)^{-1}e^{|I|} - (X^I)^{-1}\Delta_p X^I  \Delta_p z^I$ and $d^F_2 = 0$.
\State Retrieve $\Delta_c z$:
$$\Delta_c z^I = d_2^I - (X^I)^{-1}(Z^{I} \Delta_c x^{I}),\ \ \Delta_c z^F = 0.$$
\State $$(\Delta x, \Delta y, \Delta z) = (\Delta_p x + \Delta_c x, \Delta_p y + \Delta_c y, \Delta_p z + \Delta_c z).$$
\State Compute the step in the non-negativity orthant:
\begin{equation*} 
\alpha_x^{\max} = \min_{\Delta x^{I(i)} < 0} \bigg \{1,-\frac{x^{I(i)}}{\Delta x^{I(i)}}\bigg\},\ \ \alpha_z^{\max} = \min_{\Delta z^{I(i)} < 0} \bigg \{1,-\frac{z^{I(i)}}{\Delta z^{I(i)}}\bigg\},
\end{equation*}
\noindent and set:
\begin{equation*}
\alpha_x = \tau \alpha_x^{\max},\ \ \alpha_z= \tau\alpha_z^{\max}.
\end{equation*}
\State Update: 
$$(x_{k+1},y_{k+1},z_{k+1}) = (x_k+\alpha_x\Delta x,y_k + \alpha_z \Delta y,z_k + \alpha_z \Delta z).$$
\end{algorithmic}
\end{algorithm}
\par The PMM parameters are initialized as follows: $\delta_0 = 8,\ \rho_0 = 8$, $\lambda_0 = y_0$, $\zeta_0 = x_0$. At the end of every iteration, we employ the algorithmic scheme given in Algorithm \ref{Algorithm Regularization updates}. In order to ensure numerical stability, $\delta$ and $\rho$ are not allowed to become smaller than a suitable positive threshold, $\text{reg}_{thr}$. We set $\text{reg}_{thr} = \max\big\{\frac{\text{tol}}{\max\{\|A\|^2_{\infty},\|Q\|^2_{\infty}\}},10^{-13}\big\}$. This value is based on the developments in \cite{JOTA-PouGon}, where it is shown that such a constant introduces a controlled perturbation in the eigenvalues of the non-regularized linear system. If numerical instability is detected while solving the Newton system, we increase the regularization parameters ($\delta,\ \rho$) by a factor of 2 and solve the Newton system again. If this happens while either $\delta$ or $\rho$ have reached their minimum value, we also increase this threshold. If the threshold is increased 10 times, the method is terminated with a message indicating ill-conditioning. 
\renewcommand{\thealgorithm}{PEU}
\begin{algorithm}[!h]
\caption{Penalty and Estimate Updates}
    \label{Algorithm Regularization updates}
\begin{algorithmic}[!ht]
	\State $r = \frac{|\mu_{k}-\mu_{k+1}|}{\mu_k}$ (rate of decrease of $\mu$).
  	\If {($\|Ax_{k+1} - b\| \leq 0.95 \cdot \|Ax_k - b\|$)}
  		\State $\eta_{k+1} = y_{k+1}$.
  		\State $\delta_{k+1} =(1-r) \cdot\delta_k$.
  	\Else
  		\State $\eta_{k+1} = \eta_k$.
  		\State $\delta_{k+1} = (1-\frac{1}{3}r) \cdot\delta_k$.
  	\EndIf
  	\State $\delta_{k+1} = \max \{\delta_{k+1},\text{reg}_{thr} \}$, for numerical stability (ensure quasi-definiteness).
  	\If {($\|c + Qx_{k+1} - A^Ty_{k+1} -z_{k+1}\| \leq 0.95\cdot\|c+Qx_k - A^Ty_k - z_k\|$)}
  		\State $\zeta_{k+1} = x_{k+1}$.
  		\State $\rho_{k+1} = (1-r) \cdot\rho_k$.
  	\Else
  		\State $\zeta_{k+1} = \zeta_k$.
  		\State $\rho_{k+1} = (1-\frac{1}{3}r)\cdot \rho_k$.
  	\EndIf
    \State $\rho_{k+1} = \max \{\rho_{k+1},\text{reg}_{thr}\}$.
    \State $k = k+1$.
\end{algorithmic}
\end{algorithm}
\subsubsection{Preconditioner: Low-rank updates and dynamic refinement}
\par At each IP--PMM iteration we check the number of non-zeros of the preconditioner used in the previous iteration. If this number exceeds some predefined constant (depending on the number of constraints $m$), we perform certain low-rank updates to the preconditioner, to ensure that its quality is improved, without having to use very much memory. In such a case, the following tasks are performed as sketched in Algorithm \ref{BFGS0}. Then, at every Krylov iteration, the computation of the preconditioned residual $\hat{r} = \color{cyan}P^{-1} r\color{black}$ requires the steps outlined in Algorithm \ref{BFGS1}. 

{In our implementation, the first step of Algorithm \ref{BFGS0} is performed 
using the restarted Lanczos method through the inbuilt MATLAB function  \texttt{eigs}, 
requesting $1$-digit accurate eigenpairs.
This requires and additional number of applications of the preconditioned $M_{NE}$ matrix within
\texttt{eigs}. 
To save on this cost we employ the first 5 Lanczos iterations to assess the order of magnitude
of the largest eigenvalue $\lambda_{\max}(P_{NE}^{-1} M_{NE}) $. If  $\lambda_{\max} < 100$
we assume that a condition number of the preconditioned matrix less than 100 must not be
further reduced, and we stop computing eigenvalues, otherwise we proceed.  
The number of matrix-vector prodicts required
by \texttt{eigs} can not be known in advance. However, fast Lanczos convergence is expected when 
the largest eigenvalues are well separated which in turn will provide a notable reduction 
of the condition number of the preconditioned matrices. This extra cost is payed for by: (a) 
a decreased number of PCG iterations in both the predictor and corrector steps; (b) 
an improved conditioning of the linear system; (c) a saving in the density of the Cholesky
factor at the subsequent IP iteration, since, as it will explained at the end of this Section 4.2.2,
 fast PCG convergence at a given IP step will cause a sparsification of $P_{NE}$ at the next
 outer iteration.
 We finally remark that a good approximation of the largest eigenvalues of the preconditioned
 matrix could be extracted for free \cite{BFMPJcam17} during the PCG solution of the correction linear system
 and used them to accelerate the predictor linear system by the low-rank correction. This approach,
 not implemented in the present version of the code,
 would save on the cost of computing eigenpairs but would provide acceleration in the second
 linear system only.}

{
The cost of computing $Z$ is equal to $p$ matrix-vector products with matrix $M_{NE}$. Then, $T$ is computed in $O(p m)$ operations, while $\Pi$ is computed via an LU decomposition, which costs $O(p^3)$ operations. All the previous need to be calculated once before employing the Krylov subspace method. Algorithm \ref{BFGS1} introduces an additional $O(4pm)$ cost per iteration of the Krylov subspace method (notice that a similar computation as in the third step of Algorithm \ref{BFGS1} is required even without enabling low-rank updates).} {The memory requirements of Algorithms \ref{BFGS0}, \ref{BFGS1} are of the order $O(pm)$.}
\renewcommand{\thealgorithm}{LRU-0}
\begin{algorithm}[h!]
	\caption{Low-Rank Updates-0: Before the Krylov Solver Iteration}
	\label{BFGS0}
\begin{algorithmic}
	\State  Compute the $p$ rightmost (approximate) eigenvectors $v_m, \ldots, v_{m-p+1}$ 
	of  $M_{NE} v = \lambda P_{NE} v$.
	\State  Set $V = \begin{bmatrix} v_m & \ldots &  v_{m-p+1} \end{bmatrix}$
		\State  Compute             $  Z = M_{NE} V; \ T = V^T Z; \ \Pi = T^{-1}$.

\end{algorithmic}
\end{algorithm}

\vspace{-1cm}
\renewcommand{\thealgorithm}{LRU-1}

\begin{algorithm}[h!]
	\caption{Low-Rank Updates-1: Computation of $\hat r = \color{cyan}P^{-1} r\color{black}$}	\label{BFGS1}

\begin{algorithmic}
	\State  $ w = \Pi (V^T r)$.
	\State  $ z = r - Z w$.
	\State  Solve $ P_{NE} t = z $.
	\State  $ u = \Pi (Z^T t)$.
	\State  $ \hat{r} = V ( \nu w - u) + t$.
\end{algorithmic}
\end{algorithm}

\par In Section \ref{Spectral Analysis}, we showed that the quality of both preconditioners in \eqref{augmented system preconditioner} and \eqref{normal equations preconditioner} depends  heavily on the quality of the approximation of the normal equations. In other words, the quality of the preconditioner for the normal equations in \eqref{normal equations preconditioner} governs the convergence of both MINRES and CG. In turn, we know from Theorem \ref{theorem_LP}, that the quality of this preconditioner depends on the choice of the constant $C_{E,k}$, at every iteration of the method. By combining the previous with the definition of $E$ in \eqref{defE}, we expect that as $C_{E,k}$ decreases (which potentially means that there are fewer zero diagonal elements in $E$), the quality of $P_{NE,k}$ is improved. Hence, we control the quality of this preconditioner, by adjusting the value of $C_{E,k}$.
\par  More specifically, the required quality of the preconditioner depends on the quality of the preconditioner at the previous iteration, as well as on the required memory of the previous preconditioner. In particular, if the Krylov method converged fast in the previous IP--PMM iteration (compared to the maximum allowed number of Krylov iterations), while requiring a substantial amount of memory, then the preconditioner quality is lowered (i.e. $C_{E,k+1} > C_{E,k}$). Similarly, if the Krylov method converged slowly, the preconditioner quality is increased (i.e. $C_{E,k} > C_{E,k+1}$). If the number of non-zeros of the preconditioner is more than a predefined large constant (depending on the available memory), and the preconditioner is still not good enough, we further increase the preconditioner's quality (i.e. we decrease $C_{E,k}$), but at a very slow rate, hoping that this happens close to convergence (which is what we observe in practice, when solving large scale problems). As a consequence, allowing more iterations for the Krylov solvers results in a (usually) slower method that requires less memory. On the other hand, by sensibly restricting the maximum number of iterations of the iterative solvers, one can achieve fast convergence, at the expense of robustness (the method is slightly more prone to inaccuracy and could potentially require more memory).

 \subsubsection{Termination criteria}
\par The termination criteria of the method are summarized in Algorithm \ref{Algorithm Termination Criteria}. In particular, the method successfully terminates if the scaled 2-norm of the primal and dual infeasibility, as well as the complementarity, are less than a specified tolerance. The following two conditions in Algorithm \ref{Algorithm Termination Criteria} are employed to detect whether the problem under consideration is infeasible. For a theoretical justification of these conditions, the reader is referred to \cite{arxiv-PouGon}. If none of the above happens, the algorithm terminates after a pre-specified number of iterations. 

\renewcommand{\thealgorithm}{TC}
\begin{algorithm}[!h]
\caption{Termination Criteria}\label{Algorithm Termination Criteria} 

\noindent \textbf{Input:} $k$, $\text{tol}$, $\text{maximum iterations}$
\begin{algorithmic} 
	\If {$\big(\frac{\| c -A^T y + Qx - z \|}{\max\{\|c\|,1\}} \leq \text{tol}\big)$ $\wedge$ $\big(\frac{\|b - Ax\|}{\max\{\|b\|,1\}} \leq \text{tol}\big)$ $\wedge$ $\big(\mu \leq \text{tol}\big)$}
		\State Declare convergence.
	\EndIf
	\If {$\big(\|c + Qx_k - A^T y_k - z_k + \rho_k(x_k - \zeta_k)\| \leq \text{tol}\big)\  \wedge\ \big(\|x_k - \zeta_k\| > 10^{10}\big)$}
		\If {($\zeta_k$ not updated for 5 consecutive iterations)} 
			\State Declare infeasibility.
		\EndIf	
	\EndIf
	\If {$\big(\|b-Ax_k - \delta_k(y_k-\eta_k)\| \leq \text{tol}\big)\ \wedge\ \big(\|y_k - \eta_k\|> 10^{10}\big)$}
		\If {($\eta_k$ not updated for 5 consecutive iterations)}
			\State Declare infeasibility.
		\EndIf
	\EndIf

	\If {$(k > \text{iterations limit})$}
		\State Exit (non-optimal).
	\EndIf
\end{algorithmic}
\end{algorithm}

\section{Numerical Results} \label{Section Numerical Results}
\par At this point, we present computational results obtained by solving a set of small to large scale linear and convex quadratic problems.  Throughout all of the presented experiments, we set the maximum number of IP--PMM iterations to $200$.  The experiments were conducted on a PC with a 2.2GHz Intel Core i7 processor (hexa-core), 16GB RAM, run under Windows 10 operating system. The MATLAB version used was R2019a. For the rest of this section, the reported number of non-zeros of a constraint matrix of an arbitrary problem does not include possible extra entries created to transform the problem to the IP--PMM format.
\par Firstly, we run the method on the Netlib collection \cite{Netlib}. The test set consists of 96 linear programming problems. We set the desired tolerance to $\text{tol} = 10^{-4}$. In Table \ref{Netlib medium instances}, we collect statistics from the runs of the method over some medium scale instances of the Netlib test set (see \cite{Netlib}).  For each problem, two runs are presented; in the first one, we solve the normal equations of systems \eqref{Augmented System predictor}--\eqref{Augmented System corrector} using CG, while in the second one, we solve \eqref{Augmented System predictor}--\eqref{Augmented System corrector} using MINRES.  As we argued in Section \ref{Spectral Analysis}, the MINRES can require more than twice as many iterations as CG to deliver an equally good direction. Hence, we set $\text{maxit}_{\text{MINRES}} = 3 \cdot \text{maxit}_{\text{CG}} = 300$ (i.e. $\text{maxit}_{\text{CG}} = 100$). {As we already mentioned in Remark \ref{remark: cg vs minres iterations}, it is not entirely clear how many more iterations MINRES requires to guarantee the same quality of solution as PCG, since the two algorithms optimize different residual norms. Hence, requiring three times more iterations for MINRES is based on the behavior we observed through numerical experimentation.} It comes as no surprise that IP--PMM with MINRES is slower, however, it allows us to solve general convex quadratic problems for which the normal equations are too expensive to be formed, or applied to a vector ({indeed, this would often require the inversion of the matrix $Q + \Theta^{-1} + \rho_k I_n$, which is possibly non-diagonal}). 
More specifically, IP--PMM with CG solved the whole set successfully in 141.25 seconds, requiring 2,907 IP--PMM iterations and 101,382 CG iterations. Furthermore, IP--PMM with MINRES also solved the whole set successfully, requiring 341.23 seconds, 3,012 total IP--PMM iterations and 297,041 MINRES iterations.
\begin{table}[!h]
\centering
\caption{Medium Scale Linear Programming Problems\label{Netlib medium instances}}
\begin{tabular}{@{\extracolsep{4pt}}lrrrrrrrr@{}}
    \toprule
\multirow{2}{*}{\textbf{Name}}            & \multirow{2}{*}{$\bm{\textbf{nnz}(A)}$}            & \multicolumn{3}{c}{{\textbf{IP--PMM: CG}}}   & \multicolumn{3}{c}{\textbf{IP--PMM: MINRES}}\\  \cline{3-5}  \cline{6-8}
& &{Time (s)} & {IP-Iter.} & {CG-Iter.} & {Time (s)}& {IP-Iter.} & {MR-Iter.}\\ \midrule
80BAU3B &   $29,063 $ & 3.15 & 48 & 1,886 & 10.74 & 47 & 4,883 \\
D2Q06C &  $35,674$ & 2.16 & 42 & 1,562 & 7.48 & 46 & 5,080 \\
D6CUBE &  $43,888$ & 0.97 & 30 & 933 & 3.26 & 30 & 3,279 \\
DFL001 & $41,873$ & 10.18 & 54 & 2,105 & 29.07 & 54 & 6,292\\ 
FIT2D & $138,018$ & 3.16 & 28 & 836 & 10.62 & 28 & 2,558\\ 
FIT2P & $60,784$ & 40.78 & 31 & 924 & 65.15 & 31 & 2,978  \\ 
PILOT87 &  $73,804$ & 7.29 & 40 & 1,260 & 18.36 & 42 & 3,543\\
QAP12 & $44,244$ & 4.38 & 14 & 495 & 8.62 & 14 & 1,465 \\
QAP15 & $110,700$ & 22.83 & 18 & 575 & 47.45 & 18 & 1,808\\ \bottomrule
\end{tabular}
\end{table}
 
\par While we previously presented the runs of IP--PMM using MINRES over the Netlib collection, we did so only to compare the two variants. In particular, for the rest of this section we employ the convention that IP--PMM uses CG whenever $Q = 0$ or $Q$ is diagonal, and MINRES whenever this is not the case. Next, we present the runs of the method over the Maros--M\'esz\'aros test set \cite{OMS-MarMes}, which is comprised of 127 convex quadratic programming problems. In Table \ref{QP medium instances}, we collect statistics from the runs of the method over some medium and large scale instances of the collection.
\begin{table}[!h]
\centering
\caption{Medium and Large Scale Quadratic Programming Problems\label{QP medium instances}}
\begin{tabular}{@{\extracolsep{10pt}}lrrrrr@{}}
    \toprule
\multirow{2}{*}{\textbf{Name}}            & \multirow{2}{*}{$\bm{\textbf{nnz}(A)}$}            & \multirow{2}{*}{$\bm{\textbf{nnz}(Q)}$}           &\multicolumn{3}{c}{{\textbf{IP--PMM}}}   \\  \cline{4-6}  
& & &  {Time (s)}& {IP-Iter.} & {Krylov-Iter.}\\ \midrule
AUG2DCQP &20,200 & 80,400 & 4.46 & 41 & 1,188\\
CONT-100 & 49,005 & 10,197 &   3.95   & 23   & 68      \\
CONT-101 & 49,599 & 2,700 &   8.83   & 85   & 282      \\
CONT-200 & 198,005 & 40,397 &    39.84  & 109   & 422     \\
CONT-300 & 448,799 & 23,100 &    134,76  & 126   & 405     \\
CVXQP1\_L & 14,998 & 69,968 &54.77 &111 & 12,565  \\
CVXQP3\_L & 22,497 & 69,968 & 80.18 & 122 & 14,343 \\
LISWET1 & 30,000 &  10,002 & 3.55 & 41 & 1,249 \\
POWELL20 & 20,000 & 10,000 & 2.71 & 31 & 937\\ 
QSHIP12L & 16,170 & 122,433 & 2.99 & 26 & 3,312\\  \bottomrule
\end{tabular}
\end{table}

\begin{table}[!h]
\centering
\caption{Robustness of Inexact IP--PMM\label{Table robustness}}
\begin{tabular}{@{\extracolsep{4pt}}llrrrr@{}}
    \toprule
\multirow{2}{*}{\textbf{Collection}}            & \multirow{2}{*}{\textbf{Tol}}            & \multirow{2}{*}{\textbf{Solved (\%)}}           &\multicolumn{3}{c}{{\textbf{IP--PMM}}}   \\  \cline{4-6}  
& &  & {Time (s)}& {IP-Iter.} & {Krylov-Iter.}\\ \midrule
	Netlib & $10^{-4}$ & 100\phantom{.00}\% & 141.25 & 2,907 & 101,482\\
	Netlib & $10^{-6}$ & 100\phantom{.00}\%  &183.31 & 3,083 & 107,911  \\
Netlib & $10^{-8}$ & 96.87 \% &337.21 & 3,670 & 119,465  \\
Maros--M\'esz\'aros & $10^{-4}$ & 99.21 \%  & 422.75 & 3,429 & 247,724\\
Maros--M\'esz\'aros  & $10^{-6}$ & 97.64 \%  & 545.26 & 4,856 & 291,286\\
Maros--M\'esz\'aros & $10^{-8}$ &92.91 \%  & 637.35 & 5,469 & 321,636 \\ \bottomrule 
\end{tabular}
\end{table}
\begin{table}[h!]
\centering
\begin{threeparttable}

\caption{Large-Scale Linear Programming Problems\label{LP large instances}}
\begin{tabular}{@{\extracolsep{14pt}}lrrrrr@{}}
    \toprule
\multirow{2}{*}{\textbf{Name}}            & \multirow{2}{*}{$\bm{\textbf{nnz}(A)}$}            & \multicolumn{3}{c}{{\textbf{IP--PMM: CG}}}   \\  \cline{3-5}  
& & {Time (s)}& {IP-Iter.} & {CG-Iter.}\\ \midrule
CONT1-l & $7,031,999$ & $\ast$\tnote{1} & $\ast$ & $\ast$\\
FOME13 &   $285,056 $ & 72.59 & 54 & 2,098  \\
FOME21 &   $465,294 $ & 415.51 & 96 & 4,268  \\
LP-CRE-B &  $260,785$ & 14.25 & 51 & 2,177  \\
LP-CRE-D &   $246,614 $ & 16.04 & 58 & 2,516  \\
LP-KEN-18&   $358,171 $ & 128.78 & 42 & 1,759  \\
LP-OSA-30 &   $604,488 $ & 20.88 & 67 & 2,409  \\
LP-OSA-60 &   $1,408,073$ & 56.65 & 65 & 2,403  \\
LP-NUG-20&   $304,800$ & 132.41 & 17 & 785  \\
LP-NUG-30&   $1,567,800$ & 2,873.67 & 22 & 1,141  \\
LP-PDS-30&   $340,635$ & 363.89 & 81 & 3,362  \\
LP-PDS-100&   $1,096,002 $ & 3,709.93 & 100 & 6,094  \\
LP-STOCFOR3 &  $43,888$ & 8.96 & 60 & 1,777\\
NEOS & $1,526,794$ & $\dagger$\tnote{2} & $\dagger$ & $\dagger$\\ 
NUG08-3rd & $148,416$ & 80.72 & 17 & 682 \\ 
RAIL2586 & $8,011,362$ & 294.12 & 51 & 1,691  \\ 
RAIL4284 &  $11,284,032$ & 391.93 & 46 & 1,567\\
WATSON-1 & $1,055,093$ & 181.63 & 73 & 2,588 \\ 
WATSON-2 &   $1,846,391$ & 612.68 & 140 & 5,637  \\ \bottomrule
\end{tabular}
\begin{tablenotes}
\item[1] $\ast$ indicates that the solver was stopped due to excessive run time.
\item[2] $\dagger$ indicates that the solver ran out of memory.
\end{tablenotes}
\end{threeparttable}
\end{table}
\par In Table \ref{Table robustness} we collect the statistics of the runs of the method over the entire Netlib and Maros-M\'esz\'aros test sets. In particular, we solve each set with increasing accuracy and report the overall success rate of the method, the total time, as well as the total IP--PMM and Krylov iterations. All previous experiments demonstrate that IP--PMM with the proposed preconditioning strategy inherits the reliability of IP--PMM with a direct approach (factorization) \cite{arxiv-PouGon}, while allowing one to control the memory  and processing requirements of the method (which is not the case when employing a factorization to solve the resulting Newton systems). Most of the previous experiments were conducted on small to medium scale linear and convex quadratic programming problems. In Table \ref{LP large instances} we provide the statistics of the runs of the method over a small set of large scale problems. The tolerance used in these experiments was $10^{-4}$.

\par We notice that the proposed version of IP--PMM is able to solve larger problems, as compared to IP--PMM using factorization (see \cite{arxiv-PouGon}, and notice that the experiments there were conducted on the same PC, using the same version of MATLAB). {To summarize the comparison of the two approaches, we include Figure \ref{figure LP perf. prof.}. It contains the performance profiles of the two methods, over the 26 largest linear programming problems of the QAPLIB, Kennington, Mittelmann, and Netlib libraries, for which at least one of the two methods was terminated successfully. In particular, in Figure \ref{subfigure LP perf. prof. time} we present the performance profiles with respect to time, while in Figure \ref{subfigure LP perf. prof. iter} we show the performance profiles with respect to the number of IPM iterations. IP--PMM with factorization is represented by the green line (consisting	of triangles), while IP--PMM with PCG is represented by the blue line (consisting of stars). In both figures, the horizontal axis is in logarithmic scale, and represents the ratio with respect to the best performance achieved by one of the two methods, for every problem. The vertical axis shows the percentage of problems solved by each method, for different values of the performance ratio. Robustness is ``measured" by the maximum attainable percentage, with efficiency measured by the rate of increase of each of the lines (faster increase indicates better efficiency). We refer the reader to \cite{MP-DolMor} for a complete review of this benchmarking approach. As one can observe, IP--PMM with factorization was able to solve only 84.6\% of these problems, due to excessive memory requirements (namely, problems LP-OSA-60, LP-PDS-100, RAIL4284, LP-NUG-30 were not solved due to insufficient memory). As expected, however, it converges in fewer iterations for most problems that are solved successfully by both methods. Moreover, 
IP--PMM with PCG is able to solve every problem that is successfully solved by IP--PMM with factorization. Furthermore, it manages to do so requiring significantly less time, which can be observed in Figure \ref{subfigure LP perf. prof. time}. Notice that we restrict the comparison to only large-scale problems, since this is the case of interest. In particular, IP--PMM with factorization is expected to be more efficient for solving small to medium scale problems. }
\begin{figure}
\centering
\caption{Performance profiles for large-scale linear programming problems} \label{figure LP perf. prof.}
\begin{subfigure}[!ht]{0.48\textwidth}
	\includegraphics[width=\textwidth]{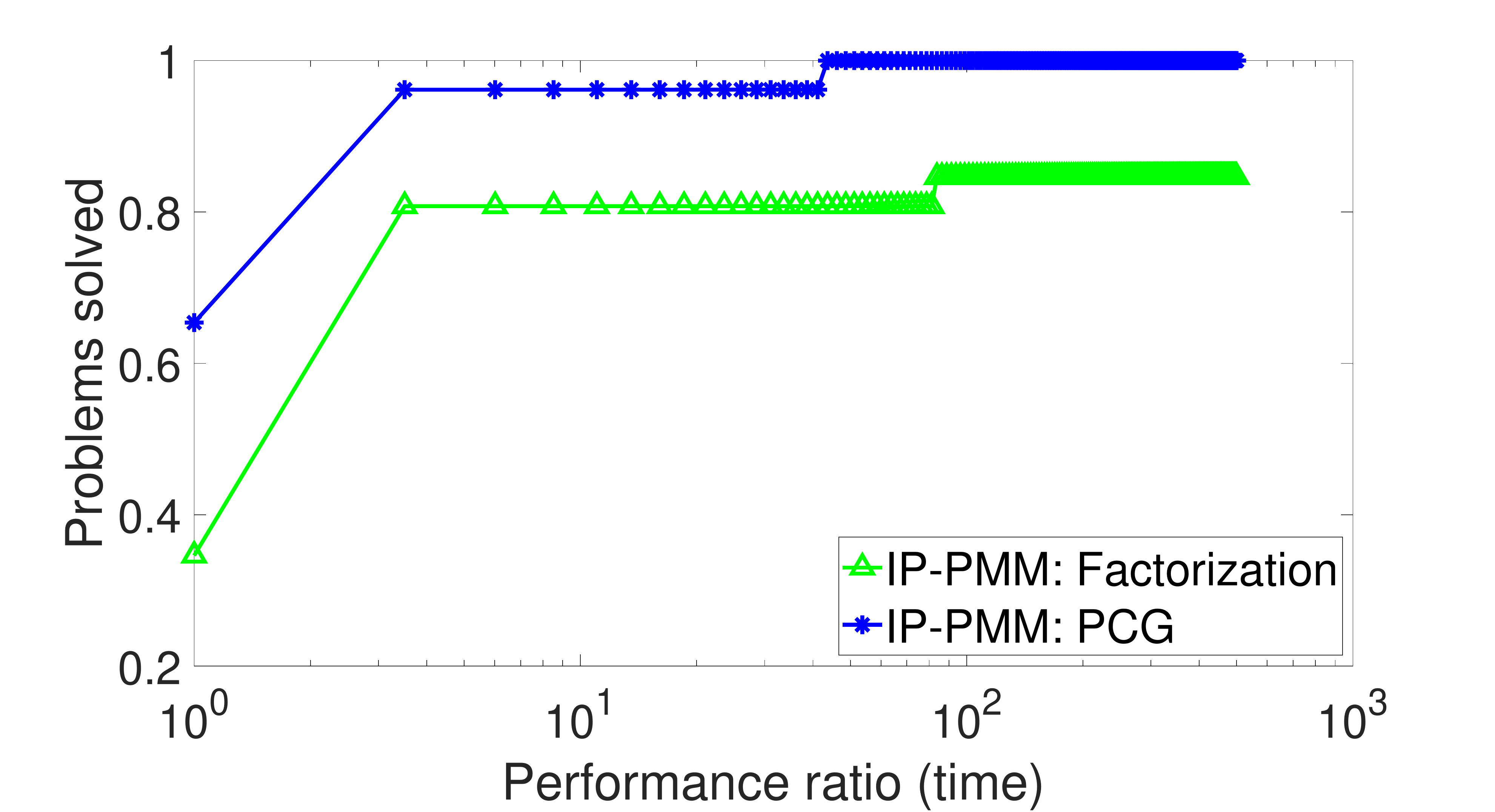}
	\caption{Performance profile in terms of CPU time}
	\label{subfigure LP perf. prof. time}
\end{subfigure}
\quad
\begin{subfigure}[!ht]{0.48\textwidth}
	\includegraphics[width =\textwidth]{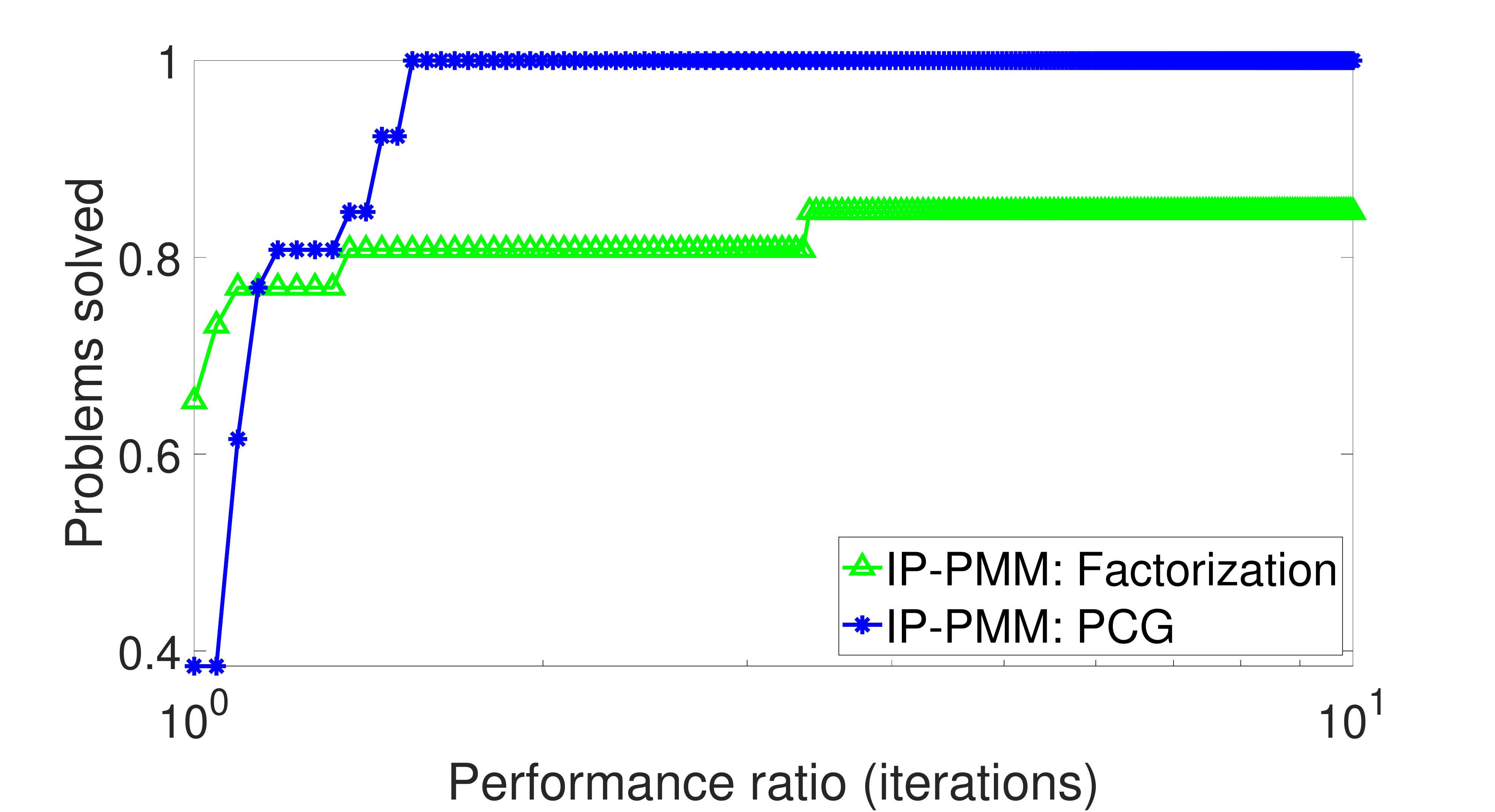}
	\caption{Performance profile in terms of iterations}
	\label{subfigure LP perf. prof. iter}
\end{subfigure}
\end{figure}

{Finally, in order to clarify the use of the low-rank (LR) updates we conducted an analysis on
two  specific -- yet representative -- linear systems, at (predictor and corrector) IP step \#12 for problem \texttt{nug20}.
In Table \ref{nug20} we report the results in solving these linear systems with the low-rank strategy and 
different accuracy/number of eigenpairs (LR($p, \texttt{tol}$) meaning  that we approximate $p$ 
eigenpairs with \texttt{eigs} with a tolerance \texttt{tol}). The best choice, using $p=10$ and  $0.1$ 
accuracy, improves the $P_{NE}$ preconditioner both in terms of linear iterations
and total CPU time.
\begin{table}[h!]
\begin{center}
	\begin{tabular}{lr|rr|rr|r}
	& & 		\multicolumn{2}{|c}{predictor} &	\multicolumn{2}{|c|}{corrector} &\\
	& CPU(\texttt{eigs}) & its & CPU & its & CPU & CPU tot\\
		\hline
		No tuning &       & 95 & 10.71 &  95 & 11.10 & 21.81 \\
		LR ($5,0.1$) &2.39 & 79 & 9.51  & 78 &  9.59& 21.49 \\ 
		LR ($10,0.1$) &3.00 & 69 & 8.14  & 67 & 7.64 & 18.78 \\ 
		LR ($20,0.1$) &5.98 & 64 & 7.79  & 63 & 7.85 & 22.62 \\ 
		LR ($20,10^{-3}$) & 9.59 & 64 & 7.79  & 63 & 7.85 & 26.23 \\ 
\end{tabular}
	\caption{CPU times and number of linear iterations for the various preconditioners at IP iteration \#12 for problem \texttt{nug20}.}
	\label{nug20}
\end{center}
\end{table}
}

{
\noindent
Figure \ref{nug20FIG} accounts for 
the steepest convergence profile of the preconditioned-with-tuning normal equations matrix,
when using the optimal parameters.
\begin{figure}[h!]
	\centerline{\includegraphics[width=8cm]{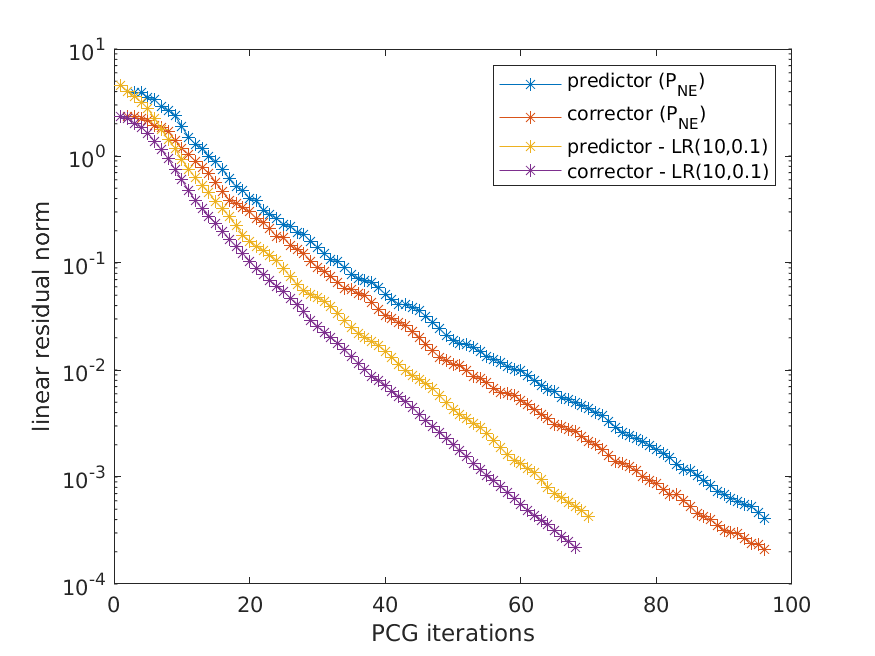}}
	\caption{Convergence profiles of PCG accelerated with $P_{NE}$ and $P_{NE}$ 
	updated with LR($10,0.1$). Linear systems at IP iteration \#12  for problem
	\texttt{nug20}.}
	\label{nug20FIG}
\end{figure}
}
\section{Concluding Remarks}\label{Conclusion}

In this paper, we have considered a combination of the interior point 
method and the proximal method
of multipliers to efficiently
solve linear and quadratic programming problems of large size. The
combined method, in short IP--PMM, produces
a sequence of linear systems whose conditioning progressively
deteriorates as the iteration proceeds.
One main contribution of this paper is the development and analysis of
a novel preconditioning technique
for both the normal equations system arising in LP and separable QP
problems, and the augmented system for general QP instances.
The preconditioning strategy consists of the construction of symmetric
positive definite, block-diagonal preconditioners for the augmented
system or a suitable approximation of the normal equations coefficient
matrix, by undertaking sparsification of the (1,1) block with the aim of controlling
the memory requirements and computational cost of the method.
We have carried out a detailed spectral analysis of the resulting
preconditioned matrix systems. In particular, we have shown that the
spectrum of the preconditioned normal equations is independent of the
logarithmic barrier parameter in the LP and separable QP cases, which is a highly
desirable property for preconditioned systems arising from IPMs. 
We have then made use of this result to obtain a spectral analysis of 
preconditioned matrix systems arising from more general QP problems.

We have reported computational results obtained by solving a set of
small to large linear and convex quadratic problems from the Netlib
and Maros--M\'esz\'aros collections, and also large-scale linear
programming problems. The experiments demonstrate that
the new solver, in conjunction with
the proposed preconditioned iterative methods, leads to rapid and
robust convergence for a wide class of problems.  We hope that this
work provides a first step towards the construction of generalizable
preconditioners for linear and quadratic programming problems.

\subsection*{Acknowledgements}
The authors express their gratitude to the reviewers for their valuable comments. 
This work was partially supported by the Project granted by the CARIPARO
foundation {\em Matrix-Free Preconditioners for Large-Scale Convex
Constrained Optimization Problems (PRECOOP)}. L. Bergamaschi and A.
Mart\'{\i}nez were also supported by the INdAM-GNCS Project (Year
2019), while J. Gondzio and S. Pougkakiotis were also supported by the
Google project {\em Fast $(1+x)$-order Methods for Linear Programming}.
We wish to remark that
this study does
not have any conflict of interest to disclose.

\setstretch{1}

\end{document}